\documentclass[11pt]{amsart} 

    \usepackage{amsthm,amssymb,amsmath}
    \usepackage{enumerate}
    \usepackage{graphicx}
    \usepackage{float}
    \usepackage{subfig}
    \usepackage{tikz}

    \theoremstyle{plain}
    \newtheorem{theorem}{Theorem}
    \newtheorem{corollary}[theorem]{Corollary}
    \newtheorem{lemma}[theorem]{Lemma}
    \newtheorem{proposition}[theorem]{Proposition}

    \theoremstyle{definition}
    \newtheorem{definition}[theorem]{Definition}
    \newtheorem{fact}[theorem]{Fact}
    \newtheorem{example}[theorem]{Example}

    \newcommand{\C}{\mathcal{C}}
    \newcommand{\Avns}{\Av_n^*(123)}
    \newcommand{\Avn}{\Av_n(123)}
    \newcommand{\Avnn}{\Av_n (132)}

    \newcommand{\D}{\mathcal{D}}
    \newcommand{\ra}{\rightarrow}
    \newcommand{\R}{\mathcal{R}}
    \newcommand{\B}{\mathcal{B}}
    \DeclareMathOperator{\Av}{Av}
    \DeclareMathOperator{\num}{f}
    \DeclareMathOperator{\Sp}{Span}

\begin{document} 

  \title{Expected Patterns in Permutation Classes}
  \author{Cheyne Homberger}
  \address{
    Cheyne Homberger, Department of Mathematics, 
    University of Florida,
    358 Little Hall,  
    Gainesville, FL 32611--8105 (USA)
  }
  \date{\today}

  \begin{abstract} 
    In the set of all patterns in $S_n$, it is clear that each
    k-pattern occurs equally often. If we instead restrict to the
    class of permutations avoiding a specific pattern, the situation
    quickly becomes more interesting. Mikl\'os B\'ona recently proved
    that, surprisingly, if we consider the class of permutations
    avoiding the pattern 132, all other non-monotone patterns of
    length 3 are equally common. In this paper we examine the class
    $\Av (123)$ of permutations avoiding $123$, and give exact
    formula for the occurrences of each length 3 pattern. While this
    class does not break down as nicely as $\Av (132)$, we find some
    interesting similarities between the two and prove that the number
    of 231 patterns is the same in each.
  \end{abstract}

  \maketitle

\section{Background}

  Let $p = p_1 p_2 \ldots p_n$ be a permutation in the symmetric group
  $S_n$ written in one-line notation. Given a permutation $q \in
  S_k$, say that $p$ contains $q$ as a pattern (denoted $q \prec p$)
  if there exist $k$ indices $ 1 \leq i_1 \leq i_2 \leq \ldots \leq
  i_k \leq n$ such that the entries $p_{i_1} p_{i_2} \ldots p_{i_k}$
  are in the same relative order as the entries of $q$ ($q_j < q_k$
  if and only if $p_{i_j} < p_{i_k}$).  If $p$ does not contain $q$ as
  a pattern, we say that $p$ \emph{avoids} $q$. 

  The set of all permutations equipped with this ordering can be
  viewed as a partially ordered set which is graded with respect to
  permutations length.  With this in mind, we define a
  \emph{permutation class} to be a downset (or ideal) of this poset.
  That is, a class is a collection of permutations $\C$ for which, if
  $p \in \C$ and $q \prec p$, then $q \in \C$. 

  Given a pattern $q$, the set $\Av (q)$ of all permutations avoiding $q$
  forms a natural permutation class, and much study has been devoted
  to understanding and enumerating classes of this form.  An early
  result in the area from Knuth \cite{knuth3}, is that the number of
  $n$-permutations avoiding the pattern $231$ is equal to the Catalan
  number $c_n = \frac{1}{n+1} \binom{2n}{n}$, and these are exactly
  the stack sortable permutations. A more comprehensive introduction
  to permutation patterns can be found in \cite{bonabook}.

  A question of Joshua Cooper and a recent result of Mik\'os B\'ona have
  opened up a new line of research: in a given permutation class, what
  can be said about the average number of occurences of each pattern?
  or equivalently: what is the total number of occurences of each
  pattern in the class?
  It is simple to show that in the class of all permutations, all
  patterns of a given length are equally common. The situation becomes
  much more interesting as we restrict our attention to smaller
  classes.

\section{Preliminaries}

  \begin{definition}
    Let $p,q$ be permutations. Denote by $\num_{q}(p)$ the number of
    occurences of $q$ in $p$ as a pattern.  
  \end{definition}

  For example, $\num_{213}(462513) = 2$ since the first third and
  fourth entries as well as the third fifth and sixth entries form
  $213$ patterns. Also, for any permutation $p$, $\num_{21}(p)$ counts
  the number of inversions of $p$. Note that every permutation
  statistic can be expressed through combinations of counts of
  permutation patterns, as described in \cite{claesson11}.
  
  We'll be concerned primarily with the total number of patterns in a
  class of permutations. For simplicity, we use similar notation. 

  \begin{definition}
    Let $\mathcal{C}$ denote a permutation class, and $q$ a
    pattern. Define $\num_{q}(\mathcal{C})= \sum_{p \in
    \mathcal{C}} \num_q(p)$.
    We will omit the $\mathcal{C}$ when the class in question is
    unambiguous.
  \end{definition}

  \begin{example} \label{linearity}  
    Let $q \in S_k$. Then it follows by linearity of expectation that
    $$\num_q(S_n) = \frac{n!}{k!} \binom{n}{k}.$$
  \end{example}
  
  The Catalan numbers will appear frequently in our enumeration, and
  so it will be useful to establish some standard notation and a few
  simple identities.

  \begin{definition}
    Let $c_n = \frac{1}{n+1} \binom{2n}{n}$ denote the $n$th Catalan
    number. Also, let 
    $$ C(x) = \sum_{n\geq 0} c_n x^n = \frac{1 - \sqrt{1-4x}}{2x}.$$
  \end{definition}

  \begin{fact}
    The following identities follow directly from the recurrence $C(x) =
    xC(x)^2 +1$.  
    $$ C(x)^2 = \frac{C(x)}{1-xC(x)} = \frac{1}{(1-xC(x))^2} 
    \text{ \  and \  } \frac{C(x)-1}{C(x)} = xC(x).$$
  \end{fact}


  In \cite{bona10} and \cite{bona12}, Mikl\'os B\'ona studied the
  class $\Avnn$ and found some surprising symmetries. He also gave
  exact formula and generating functions for the expectation of all
  length three patterns. In this paper we give a similar
  classification of the class $\Avn$, with some equally surprising
  connections to $\Avnn$. 

  Because both $132$ and $123$ are involutions and $231^{-1} = 312$,
  we have the following identity. Further identites, however, require
  considerably more effort.

  \begin{fact}
    In both $\Avnn$ and $\Avn$, $\num_{231} = \num_{312}$. 
  \end{fact}

  For a given integer $k$, inversion provides a simple bijection from
  the set of permutations containing exactly $k$ $231$ patterns to the
  set containing exactly $k$ $312$ patterns, proving not only that the
  total number of each pattern is the same, but that these numbers are
  equidistributed across the sets $\Avnn$ and $\Avn$. B\'ona showed that
  equidistribution is not required for the total number of patterns to
  be equal.

  \begin{theorem}[B\'ona] \label{bonasthm}
    In $\Avnn $, the total numbers of $231$, $213$, and $312$
    patterns are equal, and their numbers, with respect to $n$, 
    are given by the generating function
    $$ \frac{x^2C(x)^3}{(1-2xC(x))(1-4x)^\frac32}.$$
    Furthermore, $321$ is the most common pattern and $123$ is the
    least common. 
  \end{theorem}

  The first few values of $\num_q(\Avn)$ and $\num_q(\Avnn)$ for $q$
  of length $3$ are shown in the table below

  $$
  \begin{array}{c|c|c|c|c|c|c}
      \multicolumn{7}{c}{\Avn } \\
      \text{length} & \num_{123} & \num_{132} & \num_{213} 
      & \num_{231} & \num_{312} & \num_{321} \\
      \hline
      3  & 0     &    1  &    1 &    1 &    1 &    1  \\
      4  & 0     &    9  &    9 &   11 &   11 &   16  \\
      5  & 0     &    57 &   57 &   81 &   81 &  144  \\
      6  & 0     &   312 &  312 &  500 &  500 & 1016  \\ 
      7  & 0     &  1578 & 1578 & 2794 & 2794 & 6271   
    \end{array}
  $$

  \vspace{1pc}

  $$
  \begin{array}{c|c|c|c|c|c|c}
      \multicolumn{7}{c}{\Avnn } \\
      \text{length} & \num_{123} & \num_{132} & \num_{213} 
      & \num_{231} & \num_{312} & \num_{321} \\
      \hline
     3  & 1     &    0  &    1 &    1 &    1 &    1  \\
     4  & 10    &    0  &   11 &   11 &   11 &   13  \\
     5  & 68    &    0  &   81 &   81 &   81 &  109  \\
     6  & 392   &    0  &  500 &  500 &  500 &  748  \\ 
     7  & 2063  &    0  & 2794 & 2794 & 2794 & 4570   
   \end{array}
  $$
  \vspace{1pc}

  Note that $\num_{231}$ and $\num_{312}$ are \emph{not}
  equidistributed as statistics in $\Avnn$. Theorem \ref{bonasthm}
  was proved with a bijection from patterns to patterns, not
  necessarily respecting the underlying permutation.
  
  Turning our attention to the class $\Av(123)$, we will similarly
  classify the expectation of all length $3$ patterns and provide both
  generating functions and exact formula. In addition, we show some
  interesting and surprising connections to patterns in $\Avnn$. In
  particular, we will show that the number of $231$ patterns in each
  set is equal, as suggested by the numerical evidence.

\section{The class Av$(123)$}
\subsection{Patterns of length $2$}

  The simplest place to start is with patterns of length $2$. The
  number of $21$ patterns corresponds to the total number of
  inversions, and these numbers have already been studied, most
  notably in \cite{cheng7}. Clearly, the total number of $21$ patterns
  plus the number of $12$ patterns gives the total number of pairs of
  entries in all permutations in the class, which is given by
  $\binom{n}{2} c_n$. 

  \begin{theorem}[Cheng, Eu, Fu] \label{inv}
    Let $\C_n = \Avn$. Then 
    $$ \sum_{n \geq 0} f_{12}(\C_n) x^n = \frac{x^2 C(x)^2}{1-4x}.$$
    Furthermore, we have that
    $$ \num_{12}(\C_n) = 4^{n-1} - \binom{2n-1}{n}.$$
  \end{theorem}

  \begin{corollary}
    The number of $21$ patterns in the class of all $n$-permutations
    avoiding $123$ is given by
    $$ \num_{21}(\Avn) = \binom{n}{2} c_n + \binom{2n-1}{n} - 4^{n-1}.$$
  \end{corollary}

\subsection{Patterns of length $3$}

  We turn our attention now to patterns of length $3$, and provide a
  similar classification. To start, using the fact that $123$ is an
  involution provides some immediate identities, since inversion
  provides a natural map from the set $\Avn$ to itself. 

  \begin{fact} 
    In $\Avn$, $\num_{132} = \num_{213}$ and $\num_{231} =
    \num_{312}$.  
  \end{fact}

  Numerical data and intuition suggest that $\num_{132} < \num_{231}
  < \num_{321}$. We begin by establishing some basic relationships
  between these numbers which will eventually combine to give us
  exact formulae. First, note that the total number of all length $k$
  patterns is exactly $\binom{n}{k} c_n$. For $k = 3$ this gives the
  following fact.

  \begin{fact} \label{relation1}
    In the class $\Avn$ we have that
    $$ \num_{132} + \num_{213} + \num_{231} +
    \num_{312} + \num_{321} = \binom{n}{3} c_n.$$
    Note that since $\num_{132} = \num_{213}$ and
    $\num_{312} = \num_{231}$, we can rewrite this as 
    $$ 2 \num_{132} + 2 \num_{231} + \num_{321} = 
        \binom{n}{3}c_n.$$
  \end{fact}

  Our next uses Theorem \ref{inv} to provide another linear
  relationship between these three numbers.

  \begin{proposition} \label{relation2}
    In the class $\Avn$, we have
    $$ 4 \num_{132} + 2 \num_{231} = (n-2)\num_{12}.$$
  \end{proposition}
  \begin{proof}
    Rewrite the equation as 
    $$(n-2)\num_{12} - (\num_{132} + \num_{213}) =
    \num_{132}+\num_{213}+\num_{231}+\num_{312} .$$ 
    Claim that both sides count the total number of length $3$
    patterns which contain at least one $12$ pattern. Indeed, the
    right hand side counts all length $3$ patterns other than $321$.
    The left hand side first takes a $12$ pattern and adds another
    entry to it. However, this double counts each triple which
    has two $12$ patterns, and these are exactly the patterns $132$
    and $213$. Subtracting these off yields the desired identity.
  \end{proof}

  Note that we now have two linear relationships between the three
  unknown quantities, and so some new information would completely
  solve the system. We summarize this in the following lemma.

  \begin{lemma}\label{linsys}
    Let $\C = \Avn$, and let $a_n = \num_{132}(\C) = \num_{213}(\C)$, 
    $b_n = \num_{231}(\C) = \num_{312}(\C)$, and $d_n =
    \num_{321}(\C)$. 
    Then we have
    $$ \begin{array}{ll}
      2a_n + 2b_n + d_n & = \binom{n}{3} c_n \\
      4a_n + 2b_n     & = 4^{n-1} - \binom{2n-1}{n}
        \end{array}.
    $$
  \end{lemma}

  We note that Proposition \ref{relation2} has a
  complementary analogue, obtained by counting inversions instead of
  noninversions. However, this leads to a relation which is linearly
  dependent on the first two. It takes a new approach to yield new
  information, which requires a few new definitions.

  \begin{definition}
    A permutation $p = p_1p_2 \ldots p_n$ is \emph{decomposable}
    (sometimes referred to as skew-decomposable) if
    there exists $k \in [n]$ such that for all $i \leq k$ and all $j <
    k$, we have that $p_i > p_j$. 
    An \emph{indecomposable} permutation is one for which no such $k$
    exists.
  \end{definition}

  \begin{definition} 
    Denote the class of all indecomposable $123$ avoiding permutations
    by $\Av^* (123)$, and $\Av^* (123) \cap S_n$ by $\Avns$.
  \end{definition}

  In general, for simplicity of notation, indecomposability will be
  denoted with a star. Our first step, naturally, is to find the
  size of the set $\Avns$.
  
  \begin{proposition}
    For all $n \geq 1$, 
    $$|\Avns| = \frac{1}{n} \binom{2n-2}{n-1} = c_{n-1}.$$
  \end{proposition}
  \begin{proof}
    We know that $|\Avn| = c_n$, so 
    $$ \sum_{n \geq 0} |\Avn| x^n = \frac{1 - \sqrt{1 - 4x}}{2x} =
    C(x).$$ Let $C^*(x) = \sum_{n\geq 1} |\Avns| x^n$. 
    Every permutation in $\Avn$ can be expressed as a skew sum of
    indecomposable $123$ avoiding permutations, so it follows that
    $$ C(x) = 1 + C^*(x) + (C^*(x))^2 + (C^*(x))^3 + \ldots =
    \frac{1}{1-C^*(x)}.$$
    Solving this algebraically gives that $C^*(x) =
    \frac{C(x)-1}{C(x)} = xC(x)$,
    which finishes the proof.  
  \end{proof}


  Lemma \ref{linsys} now has an immediate
  indecomposable analogue, and the numbers $\num_q \left(\Avn\right)$
  and $\num_q \left( \Avns \right)$ can be related relatively easily
  for each pattern $q$.
  However, this alone does not allow us to solve for an exact formula.

  Our new information will come from exactly counting the number of
  $213$ patterns in the set $\Avns$ by building a bijection to Dyck
  paths. We start by defining these paths, which are counted by the
  Catalan numbers. 

  \begin{definition}
    A \emph{Dyck path} of length $2n$ (or of semilength $n$) is
    defined as a sequence of steps from the set $\{(1,1),(-1,1)\}$
    which begins at $(0,0)$, ends at $(2n,0)$, and never steps below
    the line $x=0$. 
  \end{definition}

  \begin{lemma} \label{biglemma}
    The generating function $A^*(x)$ for the number of $213$ patterns in
    $\Avns$ is given by 
    $$ A^*(x) = \sum_{n\geq 0} \num_{213}(\Avns) =
    \frac{x^3C(x)}{(1-4x)^{3/2}} = \frac{x^2}{2(1-4x)^{3/2}} -
    \frac{x^2}{2(1-4x)}.$$
  \end{lemma}
  \begin{proof}
    The proof consists of three parts: First, we examine the structure
    of permutations in $\Avns$, and find a simple way of counting the number
    of $213$ patterns. Second, we build a bijection onto Dyck paths
    which maps $213$ patterns to a path statistic. Finally, we find
    the weighted sum of all Dyck paths with respect to this statistic.


    Fix $n$, and let a permutation $p \in \Avns$. Note that since $p$
    avoids $123$, $p$ can be viewed as a union of two descending
    sequences, so every entry in $p$ is a left-to-right minima or a
    right-to-left maxima, and by indecomposability no entry is both.
    Graph $p$ on an $n \times n$ lattice by plotting $(i,p(i))$ for
    each $i \in [n]$, and color each left-to-right minima red and each
    right-to-left maxima blue.  Denote the sequence of red entries
    (ordered from left to right) by $\R=(r_1,r_2,\ldots r_j)$, and the
    sequence of blue entries by $\B=(b_1,b_2,\ldots b_k)$. Denote by
    $\Sp b_i$ the number of red entries below and to the left of
    $b_i$. Note that $\Sp b_i \geq 1$ for all $b_i$ by
    indecomposability. Now, we count the number of $213$ patterns in $p$. It
    follows that for any such pattern $q$, the $2$ entry and $1$ entry
    must be red, and the $3$ entry blue. It is also clear that each
    blue entry is contained in $\binom{\Sp b_i}{2}$ $213$ patterns.
    Therefore we have that 
    $$ \num_{213} (p) = \sum_{i=1}^{k} \binom{\Sp b_i}{2}.$$

    Now we are ready to build our bijection $\phi: \Avns \ra
    \D_{n-1}$, where $\D_{n-1}$ denotes the set of Dyck paths of
    semilength $n-1$. From each blue vertex, extend a vertical line to
    the $x$-axis and a horizontal line to the $y$-axis, and color each
    point of intersection of these lines green. Define a path $P'$
    from $(1,n)$ to $(n,1)$ by the following rules: 
    \begin{enumerate}[1)] 
      \item Begin by walking east from $(1,n)$ 
      \item At a blue vertex, turn south and continue walking
      \item At a green vertex, turn eash and continue walking 
      \item End at $(n,1)$ 
    \end{enumerate} 
    Rotate the path $P'$ by $\pi/4$ radians counter-clockwise to
    obtain a Dyck path $P$. This path is a slight modification of the
    path given by Krattenthaler's bijection \cite{krat01}, taking
    advantage of the indecomposability of the permutation to yield a
    more geometric description. This
    geometric interpretation of the bijection gives some additional
    insight into the number of $213$ patterns.



    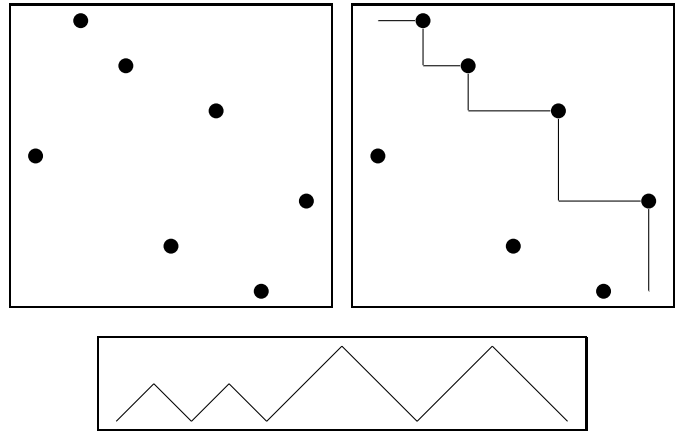
\begin{figure}[ht]
      \centering \label{permtopath}
      
      \subfloat{
      \fbox{
      \begin{tikzpicture}
      [scale=.6, auto=center, every node/.style={circle, 
      fill=black, inner sep=2pt, minimum size=2pt}]
        \foreach \x/\y in {1/4,2/7,3/6,4/2,5/5,6/1,7/3}{
          \node (p-\x) at (\x,\y) {}; }
      \end{tikzpicture} }
      }
      \subfloat{
      \fbox{
      \begin{tikzpicture}
      [scale=.6, auto=center]
        \foreach \x/\y in {1/4,4/2,6/1}{
          \node[circle,fill=black,inner sep=2pt, minimum size=2pt] 
          (r\x) at (\x,\y) {}; }
        \foreach \num/\x/\y in {1/2/7,2/3/6,3/5/5,4/7/3}{
          \node[circle,fill=black,inner sep=2pt, minimum size=2pt] 
          (b\num) at (\x,\y) {};}
        \foreach \num/\x/\y in {1/1/7,2/2/6,3/3/5,4/5/3,5/7/1}{
          \node[circle,fill=black, inner sep=0pt, minimum size=0pt] 
          (g\num) at (\x,\y) {};}
        \foreach \to/\from in {g1/b1,b1/g2,g2/b2,b2/g3,
                      g3/b3,b3/g4,g4/b4,b4/g5}{
          \draw (\to) -- (\from);}
      \end{tikzpicture} } 
      } \\

      \subfloat{
      \fbox{
      \begin{tikzpicture}
      [scale=.5, auto=center]
        \draw (0,0) -- (1,1);
        \draw (1,1) -- (2,0);
        \draw (2,0) -- (3,1);
        \draw (3,1) -- (4,0);
        \draw (4,0) -- (6,2);
        \draw (6,2) -- (8,0);
        \draw (8,0) -- (10,2);
        \draw (10,2) -- (12,0);
      \end{tikzpicture} }
      }

      \caption{$\phi(4762513) = UDUDUUDDUUDD$}
    \end{figure}

    Note that each blue entry in $p$ produces a peak in $P$.
    Furthermore, $b_i$ corresponds to a peak of height $\Sp b_i$
    above the $x$-axis in $P$. Therefore, if we let $h_{n,k}$ denote
    the total number of peaks of height $k$ in all Dyck paths
    of semilength $n$, we have that 
    $$ \num_{213}(\Avns) =\sum_{k = 1}^{n-1} \binom{k}{2} h_{n-1,k}.$$

    Finally, we can compute $H(x,u) = \sum_{n,k \geq 0} h_{n,k} x^n
    u^k$ as follows. First, note that since each Dyck path begins with
    an upstep it has a unique first point at which the path returns
    to the $x$-axis, so we can decompose each path $P$ of length $n$
    into the concatenation of two shorter paths $Q$
    and $R$. This gives that $P = uQdR$, where $u$ denotes an upstep and $d$ a
    downstep, and each peak of height $k-1$ in $Q$ and height $k$ in
    $R$ leads to a peak of height $k$ in $P$. With this in mind, we
    have the following generating function relation:
    $$ H(x,u) = ux(H(x,u)+1)C(x) + xH(x,u)C(x) .$$
    Here the first term counts the peaks from the $uQd$ part,
    including the case when $Q$ is empty. The second term counts the
    contribution from the $R$ part.  Rearranging leads to 
    $$ H(x,u) = \frac{uxC(x)}{1-uxC(x)-xC(x)}.$$

    Now, to count $213$ patterns, we need to count each peak with
    weight $\binom{k}{2}$. By taking derivatives twice with respect to
    $u$, setting $u=1$, dividing by two and scaling by $x$,
    we find that 
    $$ \begin{aligned}
    \sum_{n,k \geq 0} \binom{k}{2} h_{n-1,k}x^n &  
    = x \frac{\left. \partial_u ^2 H(x,u)\right|_{u=1}}{2} 
    = \frac{x^3C(x)}{(1-4x)^{\frac32}} \\
    & = x^3 + 7x^4 + 38x^5 + 187x^6 + 874x^7 + \ldots \ .
    \end{aligned}
    $$
        
    The sequence $0,0,1,7,38,187\ldots$ is entry A000531 in the OEIS.
    Finally, the correspondence between peaks and  $213$ patterns 
    completes the proof.
  \end{proof}

  Now, it is relatively simple to move from the class of
  indecomposable $123$ avoiding permutations to the larger class of
  all $123$ avoiding permutations. 
   
  \begin{theorem} \label{213pats}
    Let $a_n$ be the number of $213$ patterns in $\Av_n 123$. Then  
    $$ \sum_{n\geq 0} a_n x^n = \frac{x^3C(x)^3}{(1-4x)^{3/2}} = 
        \frac{x-1}{2(1-4x)} - \frac{3x-1}{2(1-4x)^{3/2}} .$$
  \end{theorem}
  \begin{proof}
    Let $A(x)$ be the generating function for the numbers $a_n$, and
    let $A^*(x)$
    denote the generating function for the number of $213$ patterns in
    \emph{indecomposable} $123$ avoiding permutations. 
    
    Now, any permutation $p$ in $\Av(123)$ can be written uniquely as a
    skew sum of a nonempty indecomposable $123$ avoiding
    permutation $q$ and another, possibly empty, $123$ avoiding
    permutation $r$. Now, it is clear that any $213$ pattern in $p$
    must be contained entirely in either $q$ or $r$. This leads to
    the following relation:
    $$ A(x) = A^*(x)C(x) + xC(x)A(x).$$
    Solving for $A$ gives
    $$ A(x) = \frac{A^*(x) C(x)}{1-xC(x)} = C^2(x) A^*(x).$$
    Lemma \ref{biglemma} now implies
    $$ A(x) = \frac{x^3 C(x)^3}{(1-4x)^{3/2}}.$$
  \end{proof}
 
  Theorem \ref{213pats} combined with Lemma \ref{linsys} allows us to
  obtain both generating functions and exact formula for the
  occurrence of all length $3$ patterns in $\Avn$. We start with $231$
  patterns, which reveal a striking connection to the class $\Av(132)$. 

  \begin{corollary} \label{231pats}
    Let $b_n$ denote the number of $231$ (or $312$) patterns in all
    $123$ avoiding $n$-permutations. Then 
    $$\sum_{n \geq 0} b_n x^n= 
    \frac{3x-1}{(1-4x)^{2}} - \frac{4x^2 - 5x + 1}{(1-4x)^{5/2}}.$$
  \end{corollary}
  \begin{proof}
    Let $B(x)$ be the generating function for the numbers $b_n$, let
    $A(x)$ be the generating function for the number of $213$ patterns,
    and let $j_n$ be the number of $12$ patterns with corresponding
    generating function $J(x)$. 
    We know from Lemma \ref{linsys} that 
    $$ 4A(x) + 2B(x) = \sum_{n \geq 0} (n-2) j_n x^n = (J(x)/x^2)' x^3
    .$$ 
    Solving this for $B(x)$ using elementary algebra and a bit of
    calculus yields
    $$ \begin{aligned}
      B(x) &=  \frac{x^2C(x)^3}{(1 - 2xC(x))(1-4x)^{3/2}} \\ 
      & = \frac{3x-1}{(1-4x)^{3/2}} - \frac{4x^2 - 5x + 1}{(1-4x)^{5/2}}
    \end{aligned}.$$
  \end{proof}

  The connection between the classes $\Av(123)$ and $\Av(132)$ is now
  immediate. 

  \begin{corollary} \label{bridge}
    Theorem \ref{231pats} together with Theorem \ref{bonasthm}
    imply immediately that the  total number of $231$ patterns in
    $\Av_n (123)$ is equal to the total number of $231$ patterns in
    $\Av_n (132)$.
  \end{corollary}
  
  We can similarly apply Lemma \ref{linsys} to $321$ patterns.
  \begin{corollary}
    Let $d_n = f_{321}(\Avn)$. Then we have that 
    $$ 
      \sum_{n\geq 0} d_n x^n = 
      \frac{ 8x^3 - 20x^2 + 8x - 1}{(1-4x)^{2}} 
      - \frac{36x^3 - 34x^2 + 10x - 1}{(1-4x)^{5/2}}.
    $$ 
  \end{corollary}
      
  Before analyzing these generating functions, we note also that Lemma
  \ref{linsys} and its indecomposable analogue produce several other
  interesting identities. We summarize some of them here for completeness.

  \begin{corollary}
    The following identities hold.
    $$ \num_{21}(\Avn) = 2\num_{213}(\Avns) $$
    $$ \num_{213}(\Avn) + \num_{231}(\Avn) = \num_{231} (\Av
    _{n-1}^*(123))$$
    $$ C(x) \left(\sum_{n\geq 0} \num_{213}(\Avn) x^n \right) = 
      x C'(x) \left(\sum_{n \geq 0} \num_{12}(\Avn)x^n \right) $$
    $$ \sum_{n\geq 0} \num_{213} (\Av_n^* (132)  x^n) = 
      \sum_{n \geq 0} \big(\num_{132}(\Avns) +
      \num_{231}(\Avns)\big)x^n. $$

  \end{corollary}

  Note that each of these identities are equivalent. That is, combined
  with Lemma \ref{linsys}, a combinatorial proof of any of them would
  imply all of the others (including Lemma \ref{biglemma}). 

  Now we can do some analysis of the main sequences. Using some
  standard generating function analysis \cite{flajolet}, we find
  that the asymptotic growth of the number of length $3$ patterns are
  as follows:

  \vspace{1pc}
  
  $$ \num_{213} (\Avn) \sim \sqrt{\frac{n}{\pi}} 4^n$$
  $$ \num_{231} (\Avn) \sim \frac{n}{2} 4^n $$
  $$ \num_{321} (\Avn) \sim \frac{8}{3} \sqrt{\frac{n^3}{\pi}} 4^n. $$

  We see that the three sequences each differ by a factor of
  approximately $\sqrt{n}$. Surprisingly, this is the same factor that
  the sequences $\num_{123}, \num_{231}, \num_{321}$ differ by in the
  class $\Av (132)$, as seen in \cite{bona10}.

  Each of these generating functions are simple enough that exact
  formulas can be obtained with relatively little hassle. One could
  argue that the asymptotic values are more interesting and
  provide more insight than the complicated formulas, but we present
  them here for completeness.

  \begin{corollary}
    Let $a_n = \num_{132}(\Avn)$, $b_n = \num_{213}(\Avn)$, and $d_n =
    \num_{321}(\Avn)$. Then we have that 
    $$ a_n = \frac{n+2}{4} \binom{2n}{n} - 3 \cdot 2^{2n-3} $$
    $$ b_n = (2n-1) \binom{2n-3}{n-2} - (2n+1)\binom{2n-1}{n-1} + 
       (n+4) \cdot 2^{2n-3}$$
    $$ \begin{aligned} d_n 
        = \frac{1}{6} \binom{2n+5}{n+1} \binom{n+4}{2} 
        &- \frac{5}{3} \binom{2n+3}{n} \binom{n+3}{2} 
        + \frac{17}{3} \binom{2n+1}{n-1} \binom{n+2}{2} \\
        &- 6\binom{2n-1}{n-2} \binom{n+1}{2} - (n+1) \cdot 4^{n-1}.
      \end{aligned}
    .$$
  \end{corollary}

  \subsection{Larger Patterns}
  
  Some of these same techniques are applicable to larger patterns.
  For example, we can easily modify Lemma \ref{linsys} to for patterns
  of all sizes. This leads to increasingy complicated expressions,
  but this simple idea can be used to prove the following proposition.

  \begin{proposition}
    Let $k, l \in \mathbb{Z}^+$, and $q$ be any permutation in $S_k$
    other than the decreasing permutation. Then for $n$ large enough,
    we have that 
    $$\num_{k \ldots 3 2 1}(\Avn) > \num_{q}(\Avn).$$
  \end{proposition}

  \begin{proof}
    Let $T$ be the set of permutation in $S_k$ which are not the
    descending permutation.  As in Fact \ref{relation1} and
    Proposition \ref{relation2},  we can express the number
    $(n-k+1)\num_{12}(\Avn)$ as a positive linear combination of
    all of $\num_{q} (\Avn)$ where $q \in T$, and we can express
    $\binom{n}{k} c_n$ as the sum of all $\num_r (\Avn)$ where $r
    \in S_n$. 
    It follows that there is a positive integer $m$ and positive
    integers $e_i$ such that
    $$ \binom{n}{k} c_n - (n-m+1) \num_{12} (\Avn) = 
      \num_{k \ldots 321} - \sum_{q \in T} e_i \num_q (\Avn).$$
    Asymptotic analysis shows that the left hand side is eventually
    positive, and so the first term on the right side eventually
    outgrows the second term, which completes the proof.
  \end{proof}

\section{Further Directions}

  The numbers $\num_q (\Av (p))$ for permutations $p,q$ exhibit numerous
  symmetries and produce many new questions. All of the generating
  presented here and in \cite{bona10} are almost rational, in the
  sense that they lie in the ring $\mathbb{Q}(x,\sqrt{1-4x})$. This
  allows for easy asymptotic analysis, and leaves open the possibility
  of bijections to other Catalan-related objects. 

  Building on what was mentioned in \cite{bona12}, we have instances
  of the same sequence of numbers which correspond sums of  statistics
  with different distributions in objects counted by the Catalan
  numbers.  Do these sequences and statistics have anologues in other
  such objects?

  Thus far, to the author's knowledge, the expectation of patterns
  has only been studied for the classes $\Av (123)$ and $\Av (132)$
  (and their symmetries). Applying these ideas to more general classes could
  yield similarly interesting identities. Note that the increasing and
  decreasing permutation do not always provide the opposite extreme
  cases: it is simple to show that $\num_{123}(\Av (2413)) =
  \num_{321}(\Av(2413))$. This leads to the natural question: in the
  set of permutations avoiding a specific pattern (or a set of
  patterns), can we easily determine what pattern is most common? And
  how large can the difference be between the most and least common?
  
  Finally, are there other occurences of the same sequence of patterns
  arising in different classes? Or within the same class, as in
  $\Avnn$? Taking this to the extreme, is there a proper, non-trivial
  permutation class for which each pattern of a given length is
  equally as common, as in the class of all permutations?

\nocite{*}
\bibliographystyle{plain}
\bibliography{expat}

\end{document}